\documentclass[a4paper,11pt]{article}

\usepackage[T1]{fontenc}
\usepackage{lmodern}

\usepackage{dsfont}

\usepackage[latin1]{inputenc}
\usepackage{amsmath}
\usepackage{amsthm}
\usepackage{amssymb}
\usepackage{mathrsfs}
\usepackage{graphicx}
\usepackage[all]{xy}
\usepackage{hyperref}

\usepackage{makeidx}

\usepackage{stmaryrd}
\usepackage{caption}

\usepackage{abstract} 

\usepackage{cleveref}

\usepackage{xcolor}

\newtheorem{thm}{Theorem}[section]
\newtheorem{cor}[thm]{Corollary}
\newtheorem{claim}[thm]{Claim}

\newtheorem{lemma}[thm]{Lemma}
\newtheorem{prop}[thm]{Proposition}

\theoremstyle{definition}
\newtheorem{definition}[thm]{Definition}

\def\rquotient#1#2{%
	\makeatletter
	\raise.3ex\hbox{$#1$}/\lower.3ex\hbox{$#2$}%
	\makeatother
}	

\makeatletter
\newcommand{\subjclass}[2][2010]{%
	\let\@oldtitle\@title%
	\gdef\@title{\@oldtitle\footnotetext{#1 \emph{Mathematics subject classification.} #2}}%
}
\newcommand{\keywords}[1]{%
	\let\@@oldtitle\@title%
	\gdef\@title{\@@oldtitle\footnotetext{\emph{Key words and phrases.} #1.}}%
}
\makeatother

\newcommand{\Address}{{
		\bigskip
		\small
		
\noindent\textsc{UCLouvain\\ 
Institut de recherche en Math\'ematiques et physique\\
 Chemin du Cyclotron 2\\
1348 Louvain-la-Neuve (Belgium)}\par\nopagebreak
\noindent\textit{E-mail address}: \texttt{oussama.bensaid@uclouvain.be}
  
  	\bigskip
		\small
		
\noindent\textsc{University of Montpellier\\ 
Institut Math\'ematiques Alexander Grothendieck\\
Place Eug\`ene Bataillon\\
34090 Montpellier (France)}\par\nopagebreak
\noindent\textit{E-mail address}: \texttt{anthony.genevois@umontpellier.fr}
  
    \bigskip
		\small
		
\noindent\textsc{University of Paris-Cit\'e\\ 
Institut de Math\'ematiques de Jussieu-Paris Rive Gauche\\
Place Aur\'elie Nemours\\
75013 Paris (France)}\par\nopagebreak
\noindent\textit{E-mail address}: \texttt{romain.tessera@imj-prg.fr}		
}}

\makeindex

\title{Virtual splittings of right-angled Artin groups}
\date{\today}
\author{Oussama Bensaid, Anthony Genevois, and Romain Tessera}


\begin{document}

\maketitle

\begin{abstract}
In this article, we determine, given a finite graph $\Gamma$ and an integer $n \geq 1$, when a right-angled Artin group $A(\Gamma)$ virtually splits over an abelian subgroup of rank $n$. More precisely, we show that the following assertions are equivalent: (1) $A(\Gamma)$ admits $\mathbb{Z}^n$ as a codimension-one subgroup, (2) $A(\Gamma)$ virtually splits over $\mathbb{Z}^n$, (3) $A(\Gamma)$ splits over $\mathbb{Z}^n$, and (4) $\Gamma$ either is a complete graph with $n+1$ vertices or contains a complete subgraph of size $n$ that has a subgraph separating $\Gamma$. Our result improves the main result of \cite{MR3954281}. 
\end{abstract}

\tableofcontents

\section{Introduction}

\noindent
Given a graph $\Gamma$, the \emph{right-angled Artin group} $A(\Gamma)$ is given by the presentation
$$\langle u \text{ vertex of } \Gamma \mid [u,v]=1 \text{ whenever $u$ and $v$ are adjacent in $\Gamma$} \rangle.$$
Right-angled Artin groups interpolate between free groups (when ``nothing commutes'', i.e.\ $\Gamma$ has no edges) and free abelian groups (when ``everything commutes'', i.e.\ $\Gamma$ is a complete graph). {Despite the early appearance of similar interpolations for other algebraic structures, such as monoids and algebras (see for instance \cite{MR239978, MR575780, Bergman}), as well as a few mentions of the groups in some articles \cite{MR300879, MR567067, MR629331}, the study of right-angled Artin groups  started in earnest with the work of Baudisch \cite{MR463300, MR634562} and Droms \cite{MR2633165, MR880971, MR891135, MR910401}.} Since then, many articles have been dedicated to right-angled Artin groups from various perspectives. In addition to being an instructive source of examples, right-angled Artin groups also led to interesting applications. Two major contributions in this direction include the Morse theory developed in \cite{MR1465330} and the theory of special cube complexes introduced in \cite{MR2377497}. 

\medskip \noindent
Despite the fact that right-angled Artin groups have been extensively studied, a few very natural questions remain completely open:
\begin{description}
	\item[(Embedding Problem)] Given two graphs $\Gamma_1$ and $\Gamma_2$, when is $A(\Gamma_1)$ isomorphic to a subgroup of $A(\Gamma_2)$?
	\item[(Commensurability Problem)] Given two graphs $\Gamma_1$ and $\Gamma_2$, when are $A(\Gamma_1)$ and $A(\Gamma_2$ (abstractly) commensurable, i.e.\ when do they share isomorphic finite-index subgroups?
	\item[(Quasi-isometric Problem)] Given two graphs $\Gamma_1$ and $\Gamma_2$, when are $A(\Gamma_1)$ and $A(\Gamma_2)$ quasi-isometric?
	\item[(Surface Subgroup Problem)] Given a graph $\Gamma$, when does $A(\Gamma)$ have a subgroup isomorphic to the fundamental group of a closed surface of genus $\geq 2$?
\end{description}
In this article, we are mainly interested in the Commensurability Problem, which is of course closely related to the Quasi-isometric Problem since commensurable finitely generated groups are automatically quasi-isometric. The Commensurability Problem has been solved only for a few families of graphs. An interesting playground is given by trees, because right-angled Artin groups defined by finite trees of diameter $\geq 3$ are known to be all quasi-isometric \cite{MR2376814}. So far, the Commensurability Problem has been solved only for trees of diameter $\leq 4$ \cite{MR3945733} and paths \cite{MR4243768}. Aside trees, notable examples are given by results of quasi-isometric rigidity \cite{MR2421136,HuangI,HuangII}, including right-angled Artin groups with finite outer automorphism groups. Several quasi-isometric and commensurability invariants are also worth mentioning, including divergence \cite{MR2874959, MR3073670}, cut-points in asymptotic cones \cite{MR2874959}, extension graphs for some graphs of large girth \cite[Corollary~8]{MR3192368}, Morse boundaries \cite{MR4563334}, and some homotopy types of simplicial complexes \cite{Bowtie}. 

\medskip \noindent
More precisely, we are interested in the following question: when does a right-angled Artin group virtually split over an abelian subgroup, i.e.\ when does it contain a finite-index subgroup that decomposes non-trivially as an amalgamated product or an HNN extension over an abelian subgroup? For splittings over trivial subgroups, i.e.\ free product decompositions, one knows from Stallings' theorem that admitting such a splitting is preserved by quasi-isometry, and in particular by commensurability, and it is known that a given right-angled Artin group splits non-trivially as a free product if and only if the graph that defines it is disconnected. For splittings over infinite cyclic subgroups, one also knows from \cite{MR2153400} that admitting such a splitting is preserved by quasi-isometry, and in particular by commensurability, and \cite{MR3278386} shows that a given right-angled Artin group splits non-trivially over an infinite cyclic group if and only if the graph that defines it has a cut-vertex. For splittings over abelian subgroups of arbitrary rank, it follows from \cite{CoarseRAAG} that, among finitely generated right-angled Artin groups, admitting such a splitting is again preserved by quasi-isometry, and in particular by commensurability, which happens precisely when the graph defining the right-angled Artin group under consideration has a separating complete subgraph \cite{MR3728497}. But what about virtual splittings over an abelian subgroup with a fixed rank $\geq 2$? 

\medskip \noindent
\begin{minipage}{0.3\linewidth}
\begin{center}
\includegraphics[width=0.8\linewidth]{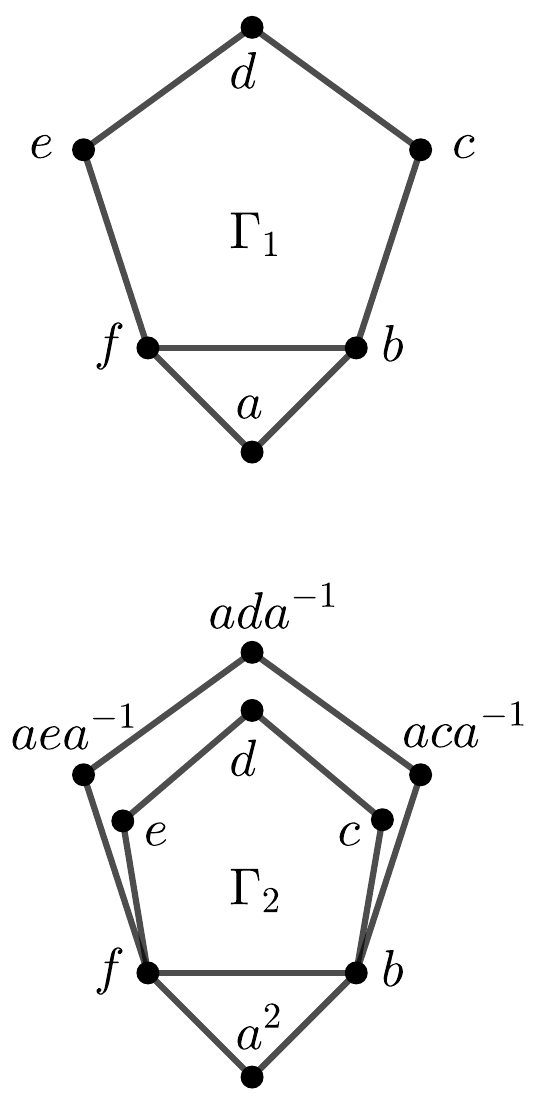}
\end{center}
\end{minipage}
\begin{minipage}{0.68\linewidth}
It is worth mentioning that the naive guess claiming that a right-angled Artin group $A(\Gamma)$ virtually splits over an abelian group of rank $n$ if and only if $\Gamma$ can be separated by a complete subgraph of size $n$ does not hold. For instance, if $P_2$ denotes the path of length two, then $A(P_2) \simeq \mathbb{F}_2 \times \mathbb{Z}$ (virtually) splits over $\mathbb{Z}^2$, since $\mathbb{F}_2$ splits over $\mathbb{Z}$, despite the fact that removing an edge (with its endpoints) cannot disconnected $P_2$. As another example, let $\Gamma_1$ and $\Gamma_2$ be the two graphs illustrated on the left. Clearly, $\Gamma_1$ does not contain a separating complete subgraph of size $3$, but $\Gamma_2$ does. Nevertheless, $A(\Gamma_1)$ contains $A(\Gamma_2)$ as a finite-index subgroup (namely, $\langle a^2,b,c,d,e,aba^{-1},aca^{-1},ada^{-1},aea^{-1}\rangle$), and the latter clearly splits over $\mathbb{Z}^3$. In fact, it can be shown that $A(\Gamma_1)$ does not only virtually splits over $\mathbb{Z}^3$, but also splits over $\mathbb{Z}^3$ directly (see Lemma~\ref{lem:NonObviousSplit}). 
\end{minipage}

\medskip \noindent
The first work to investigate virtual splittings over abelian subgroups is \cite{MR3954281}, where it is proved that, if a right-angled Artin group $A(\Gamma)$ virtually splits over an abelian subgroup of rank $n$, then $\Gamma$ must be separated by some complete subgraph with $\leq n$ vertices. The main result of our article is the following characterisation:

\begin{thm}\label{thm:BigIntro}
Let $\Gamma$ be a finite graph and $n \geq 0$ an integer. The following assertions are equivalent:
\begin{itemize}
	\item[(i)] $A(\Gamma)$ contains a coarsely separating abelian subgroup of rank $n$;
	\item[(ii)] $A(\Gamma)$ virtually splits over an abelian subgroup of rank $n$;
	\item[(iii)] $A(\Gamma)$ splits over an abelian subgroup of rank~$n$;
	\item[(iv)] $\Gamma$ either is a complete graph with $n+1$ vertices or contains a complete subgraph of size $n$ that has a subgraph separating~$\Gamma$.
\end{itemize}
\end{thm}

\noindent
Here, given a finitely generated group $G$, a subgroup $H \leq G$ is \emph{coarsely separating} if some neighbourhood of $H$ in (a Cayley graph of) $G$ has at least two deep connected components (i.e.\ two components with points arbitrarily far away from $H$). Examples of coarsely separating include codimension-one subgroups, i.e.\ subgroups with $\geq 2$ relative ends (as defined in \cite{MR357679}). In fact, since a subgroup on which there is a splitting always has codimension-one, the conditions $(i)-(iv)$ from Theorem~\ref{thm:BigIntro} are also equivalent to $A(\Gamma)$ containing a codimension-one abelian subgroup of rank $n$. (Actually, it turns out that an abelian subgroup is coarsely separating if and only if it is virtually a codimension-one subgroup. We refer the reader to \cite{CodimCoarseSep} and references therein for more information on the similarities and differences between codimension-one and coarsely separating subgroups.) 

\medskip \noindent
As an immediate consequence of Theorem~\ref{thm:BigIntro}, one obtains the following commensurability invariant:

\begin{cor}
Let $\Gamma_1$ and $\Gamma_2$ be two finite graphs such that $A(\Gamma_1)$ and $A(\Gamma_2)$ are commensurable. If $\Gamma_1$ contains a complete subgraph of size $n$ that has a subgraph separating $\Gamma_1$, then $\Gamma_2$ must also contain such a subgraph. 
\end{cor}

\noindent
Notice that, since the condition (iv) from Theorem~\ref{thm:BigIntro} implies that our graph $\Gamma$ must contain a separating complete subgraph with $\leq n$ vertices, one recovers the main result of \cite{MR3954281} from Theorem~\ref{thm:BigIntro}.

\subsection*{Acknowledgements}
The first-named author acknowledges support from the FWO and F.R.S.-FNRS under the Excellence of Science (EOS) programme (project ID 40007542).

\section{Right-angled Artin groups}

\noindent
This section is dedicated to the following observation about right-angled Artin groups, which will play an important role in the proof of Theorem~\ref{thm:BigIntro}. 

\begin{prop}\label{prop:ArtinSubAbelian}
Let $\Gamma$ be a graph and $\Lambda \leq \Gamma$ a complete subgraph of size $m$. If $\langle \Lambda \rangle$ is virtually contained in an abelian group of rank $n \geq m$, then $\Lambda$ is contained in a complete subgraph of size $n$. 
\end{prop}

\noindent
Here, given a graph $\Gamma$ and a subgraph $\Lambda \leq \Gamma$, we denote by: 
\begin{itemize}
	\item $\mathrm{link}(\Lambda)$ the subgrapb of $\Gamma$ induced by the vertices that are adjacent to all the vertices of $\Lambda$;
	\item $\mathrm{star}(\Lambda)$ the subgraph of $\Gamma$ induced by $\Lambda\cup \mathrm{link}(\Lambda)$;
	\item $\langle \Lambda \rangle$ the subgroup of $A(\Gamma)$ generated by the vertex-set $V(\Lambda)$ of $\Lambda$ (thought of as generators of $A(\Gamma)$). 
\end{itemize}
Our proposition will be an easy consequence of the next two elementary facts about right-angled Artin groups. 

\begin{lemma}\label{lem:ArtinNormaliser}
Let $\Gamma$ be a graph and $\Lambda \leq \Gamma$ a subgraph. The commensurator of $\langle \Lambda \rangle$ in $A(\Gamma)$ coincides with its normaliser, namely $\langle \mathrm{star}(\Lambda) \rangle$.
\end{lemma}

\begin{proof}
Let $g \in \mathrm{Comm}(\langle \Lambda \rangle)$, i.e.\ $g \langle \Lambda \rangle g^{-1} \cap \langle \Lambda \rangle$ has finite index in $\langle \Lambda \rangle$. We know from \cite[Proposition~3.4]{MR3365774} that there exist $h \in \langle \Lambda \rangle$ and $\Xi \leq \Lambda$ such that $g \langle \Lambda \rangle g^{-1} \cap \langle \Lambda \rangle= h \langle \Xi \rangle h^{-1}$. Under the abelianisation map $A(\Gamma) \to \mathbb{Z}^{V(\Gamma)}$, $h \langle \Xi \rangle h^{-1}$ is sent to $\mathbb{Z}^{V(\Xi)}$, which must have finite index in the image $\mathbb{Z}^{V(\Lambda)}$ of $\langle \Lambda \rangle$. The only possibility is that $\Xi = \Lambda$. Hence
$$g \langle \Lambda \rangle g^{-1} \cap \langle \Lambda \rangle  = h \langle \Xi \rangle h^{-1}= h \langle \Lambda \rangle h^{-1}= \langle \Lambda \rangle,$$
i.e.\ $\langle \Lambda \rangle \leq g \langle \Lambda \rangle g^{-1}$. By applying the same argument to $g^{-1} \in \mathrm{Comm}(\langle \Lambda \rangle)$ shows that $\langle \Lambda \rangle \leq g^{-1} \langle \Lambda \rangle g$, or equivalently that $g \langle \Lambda \rangle g^{-1} \leq  \langle \Lambda \rangle$. We conclude that $g \langle \Lambda \rangle g^{-1}= \langle \Lambda \rangle$, i.e.\ $g$ normalises $\langle \Lambda \rangle$. 

\medskip \noindent
Thus, we have proved that the commensurator of $\langle \Lambda \rangle$ in $A(\Gamma)$ coincides with its normaliser. The description of this normaliser is given by \cite[Proposition~3.13]{MR3365774}.
\end{proof}

\begin{lemma}\label{lem:ArtinSubAbelian}
Let $\Gamma$ be a graph and $n \geq 0$ an integer. The right-angled Artin group $A(\Gamma)$ contains an abelian subgroup of rank $n$ if and only if $\Gamma$ contains a complete subgraph of size $n$. 
\end{lemma}

\begin{proof}
If $\Gamma$ contains a complete subgraph $\Lambda$ with $n$ vertices, then $\langle \Lambda \rangle$ provides an abelian subgroup of rank $n$ of $A(\Gamma)$. Conversely, if $\Gamma$ does not contain a complete subgraph with $n$ vertices, then its Salvetti complex has dimension $\leq n-1$, which implies that $A(\Gamma)$ has cohomological dimension $\leq n-1$, and consequently cannot contain an abelian subgroup of rank $n$. 
\end{proof}

\begin{proof}[Proof of Proposition~\ref{prop:ArtinSubAbelian}.]
By assumption, there exists a monomorphism $\mathbb{Z}^n \hookrightarrow \mathrm{Comm}(\langle \Lambda \rangle)$. According to Lemma~\ref{lem:ArtinNormaliser}, the commensurator of $\langle \Lambda \rangle$ is $\langle \mathrm{star}(\Lambda) \rangle$. Thus, we find a morphism
$$\varphi : \mathbb{Z}^n \hookrightarrow \mathrm{Comm}(\langle \Lambda \rangle) = \langle \Lambda \rangle \oplus \langle \mathrm{link}(\Lambda) \rangle \twoheadrightarrow \langle \mathrm{link}(\Lambda) \rangle,$$
where the second arrow is the canonical projection. Since the kernel of $\varphi$ has rank at most $|V(\Lambda)| = m$, its image must be a free abelian group of rank $\geq n-m$. It follows from Lemma~\ref{lem:ArtinSubAbelian} that $\mathrm{link}(\Lambda)$ contains a complete subgraph $\Xi$ of size $n-m$. We conclude, as desired, that $\Lambda$ is contained in a complete subgraph of size $n$, namely the subgraph induced by $\Lambda \cup \Xi$. 
\end{proof}

\section{Complete cuts}

\noindent
Splittings of a right-angled Artin group $A(\Gamma)$ over abelian subgroups are closely related to separating complete subgraphs of $\Gamma$. For convenience, we introduce the following definition:

\begin{definition}
Let $X$ be a graph. A \emph{complete cut} is a complete subgraph $Y \leq X$ such that $X \backslash Y$ is disconnected.
\end{definition}

\noindent
A key ingredient in the proof of Theorem~\ref{thm:BigIntro} is that, given a right-angled Artin group $A(\Gamma)$, it is possible to decompose $\Gamma$ along complete cuts in order to construct a convenient decomposition of $A(\Gamma)$ as a graph of groups with abelian-edge-groups. The following definition records the decomposition of $\Gamma$ that will be relevant for us:

\begin{definition}
A \emph{complete-cut-decomposition} $(T,(V_s)_{s \in V(T)})$ of a graph $X$ is the data of a tree $T$ and a collection of induced subgraphs $V_s \leq X$ indexed by $V(T)$ such that:
\begin{itemize}
	\item $\{V_s \mid s \in V(T)\}$ covers $X$, i.e.\ $E(X)= \bigcup_{s \in V(T)} E(V_s)$;
	\item for every $s \in V(T)$, $V_s$ has no complete cut;
	\item for all adjacent $r,s \in V(T)$, $V_r \cap V_s$ is a complete cut of $X$ properly contained in both $V_r$ and $V_s$.
\end{itemize}
\end{definition}

\noindent
We conclude this section by proving that every finite graph admits such a nice decomposition. 

\begin{prop}\label{prop:CCD}
Every finite graph admits a complete-cut-decomposition.
\end{prop}

\begin{proof}
Let $X$ be a finite graph. If $X$ has no complete cut, then $(\{\mathrm{pt}\}, (X))$ is a (trivial) complete-cut-decomposition of $X$. Otherwise, fixing be a complete cut $C$ of minimal size, we decompose $X$ as the union of two subgraphs containing $C$ properly, say $X_1$ and $X_2$, amalgamated along $C$. Arguing by induction on the number of vertices, we know that, for $i=1,2$, $X_i$ admits a complete-cut-decomposition $(T_i, (V_s^i)_{s \in V(T)})$. 

\begin{claim}\label{claim:CCD}
For $i=1,2$, there exists $o(i) \in V(X_i)$ such that $C \subsetneq V_{o(i)}^i$. 
\end{claim}

\noindent
Fix an $i=1,2$. If the complete-cut-decomposition of $X_i$ is trivial, i.e.\ $T_i$ is reduced to a single vertex $p$ and $V_p=X_i$, then it suffices to set $o(i):=p$. From now on, we assume that the complete cut-decomposition of $X_i$ is non-trivial.

\medskip \noindent
Given an edge $e:=\{r,t\}$ of $T_i$, if $D_1$ and $D_2$ denote the two halfspaces of $T_i$ delimited by $e$ and if 
$$C \subset \left( \bigcup\limits_{s \in V(D_1)} V_s^i \right) \cap \left( \bigcup\limits_{s \in V(D_2)} V_s^i \right),$$
then $C \subset V_r^i \cap V_t^i \subsetneq V_r^i$. In this case, we can set $o(i):=r$. So, from now on, we assume that each edge $e$ of $T_i$ delimits a unique halfspace $D_e$ such that $C \subset \bigcup_{s \in V(D_s)} V_s^i$. The intersection
$$D:= \bigcap\limits_{e \text{ edge of } T_i} D_e$$
is either empty or a single vertex. In the former case, i.e.\ $D=\emptyset$, there must exist two edges of $T_i$, say $e$ and $f$, such that $D_e \cap D_f= \emptyset$. In fact, up to replacing $e$ and $f$ with two consecutive edges along the geodesic connecting $e$ and $f$ in $T_i$, we can assume that $e$ and $f$ share an endpoint, say $q$. Then, if $p$ (resp.\ $r$) denotes the other endpoint of $e$ (resp.\ $f$), we have
$$C \subset \left( \bigcap\limits_{s \in V(D_e)} V_s^i \right) \cap \left( \bigcap\limits_{s \in V(D_f)} V_s^i \right) \subset V_p \cap V_r \subsetneq V_q.$$
Thus, we can set $o(i):=q$ in this case. From now on, assume that $D$ is reduced to a single vertex, say $r$. Then, $C \subset V_r$. Notice that, because we know that the complete-cut-decomposition of $X_i$ is non-trivial, $r$ must have a neighbour in $T_i$, say $t$. It follows that, if $C=V_r$, then $V_r \cap V_t$ is a complete cut of $X$ that is smaller than $C$, contradicting our choice of $C$. Thus, $C \subsetneq V_r$, so we can set $o(i):=r$. This concludes the proof of Claim~\ref{claim:CCD}.  

\medskip \noindent
Thanks to Claim~\ref{claim:CCD}, we can combine the complete-cut-decompositions of $X_1$ and $X_2$ into a decomposition of $X$. For this, let $T$ denote the tree obtained from the disjoint union of $T_1$ and $T_2$ by connecting $o(1)$ and $o(2)$ with an edge; for every vertex $s \in V(T)$, we set $V_s:=V_s^1$ if $s \in V(T_1)$ and $V_s^2$ if $s \in V(T_2)$. We claim that $(T, (V_s)_{s \in V(T)})$ is a complete-cut-decomposition of $X$.

\medskip \noindent
First, it is clear that $\{V_s \mid s \in V(T)\}$ covers $X$. Indeed, every edge of $X$ is an edge of $X_1$ (and consequently of some $V_s^1$) or an edge of $x_2$ (and consequently of some $V_s^2$). Next, it is clear that each $V_s$ has no complete cut, since we already know that this is true for the $V_s^1$ and the $V_s^2$. It is also clear that, given two adjacent vertices $r,s \in V(T)$, $V_r \cap V_s \subsetneq V_r$. Indeed, we already know that this is true if $r,s \in V(T_1)$ or if $r,s \in V(T_2)$, and the only remaining case is $\{r,s\} = \{o(1), o(2)\}$, but we know that $V_{o(1)} \cap V_{o(2)} = C$ is properly contained in both $V_{o(1)}^1$ and $V_{o(2)}^2$ (as a consequence of our choice for $o(1)$ and $o(2)$, see Claim~\ref{claim:CCD}). 

\medskip \noindent
In order to conclude the proof of our proposition, it remains to verify that, given two adjacent vertices $r,s \in V(T)$, $V_r \cap V_s$ is a complete cut of $X$. If $\{r,s\}=\{o(1),o(2)\}$, then $V_r\cap V_s =C$ is a complete cut of $X$. Otherwise, $r,s \in V(T_i)$ for some $i=1,2$. We know that $V_r \cap V_s$ is a complete cut of $X_i$, so we can decompose $X_i$ as a union of subgraphs containing $V_r \cap V_s$ properly, say $A_1, \ldots, A_n$ ($n \geq 2$), amalgamated along $V_r \cap V_s$. Necessarily, $C$ must be contained in some $A_j$, say $A_1$ up to reindexing our subgraphs. Then $V_r \cap V_s$ separates $A_2 \backslash (V_r \cap V_s)$ and $X_2 \backslash C$ in $X$, proving that $V_r \cap V_s$ is a complete cut of $X$, as desired. 
\end{proof}

\section{Coarse separation}

\noindent
In this section, we record basic definitions and properties about coarse separation that will be useful for the proof of Theorem~\ref{thm:BigIntro}. We start by recalling the definition of coarse separation used in \cite{CoarseSep} (restricted to connected graphs for simplicity).

\begin{definition}
Let $X$ be a connected graph and $\mathcal{Z}$ a collection of subgraphs. A connected subgraph $Y \subset X$ is \emph{coarsely separated by $\mathcal{Z}$} if there exists $L \geq 0$ such that, for every $D \geq 0$, there is some $Z \in \mathcal{Z}$ such that $Y \setminus Z^{+L}$ has at least two connected components with points at distance $\geq D$ from $Z$. 
\end{definition}

\noindent
Understanding coarse separation in Euclidean spaces will be important in order to prove Theorem~\ref{thm:BigIntro}. We start by recording two elementary observations.

\begin{lemma}\label{lem:CodimAb}
For every $n \geq 1$, a subgroup of $\mathbb{Z}^n$ is coarsely separating if and only if its rank is $n-1$. 
\end{lemma}

\begin{proof}
Let $H \leq \mathbb{Z}^{n}$ be a subgroup. First, assume that $H$ has rank $n$. In other words, $H$ has finite index in $\mathbb{Z}^n$. As a quasi-dense subspace, it cannot be coarsely separating. Next, assume that $H$ has rank $n-1$. Then, $\mathbb{Z}^n/H$ is virtually infinite cyclic, hence two-ended. The preimage under the quotient map $\mathbb{Z}^n \twoheadrightarrow \mathbb{Z}^n / H$ of a ball separating $\mathbb{Z}^n/H$ into at least two unbounded connected components yields a neighbourhood of $H$ that separates $\mathbb{Z}^n$ into at least two deep connected components (i.e.\ components with point arbitrarily far away from $H$). Thus, $H$ is coarsely separating. Finally, assume that $H$ has rank $\leq n-2$. Fix a basis $\{a_1, \ldots, a_k\}$ of $H$. Because coarse separation is preserved under quasi-isometry, $H$ coarsely separates $\mathbb{Z}^n$ if and only if the vector subspace $V$ spanned by $\{a_1,\ldots, a_k\}$ coarsely separates $\mathbb{R}^n$. But $V$ has dimension $\leq n-2$, so such a coarse separation cannot occur. Thus, $H$ is not coarsely separating. 
\end{proof}

\begin{lemma}\label{lem:CoarseSepEuclidean}
Let $n \geq 1$ be an integer and $Y$ a subspace of $\mathbb{Z}^n$ contained in $\mathbb{Z}^{n-1} \times \{0\}$. If $Y$ coarsely separates $\mathbb{Z}^n$, then it must be quasi-dense in $\mathbb{Z}^{n-1} \times \{0\}$. 
\end{lemma}

\begin{proof}
Assume that $Y$ is not quasi-dense in $\mathbb{Z}^{n-1} \times \{ 0\}$. Our goal is to show that, for every $L \geq 0$, $\mathbb{Z}^n\backslash Y^{+L}$ is connected. Notice that 
$$\bigcup\limits_{y \in (\mathbb{Z}^{n-1} \times \{0\})\backslash Y^{+L}} \{y\} \times \mathbb{Z}$$
is disjoint from $Y^{+L}$. Since $(\mathbb{Z}^{n-1} \times \{0\})\backslash Y^{+L}$ is non-empty, as $Y$ is not quasi-dense in $\mathbb{Z}^{n-1} \times \{0\}$, it clearly follows that $\mathbb{Z}^n \backslash Y^{+L}$ is connected, as desired. 
\end{proof}

\noindent
Our next statement is more interesting, and follows from \cite[Theorem~1.5]{CoarseSep}.

\begin{thm}[\cite{CoarseSep}]\label{thm:SepEucGrowth}
If $\mathbb{Z}^n$ is coarsely separated by some family $\mathcal{Z}$, then $\mathcal{Z}$ has growth at least polynomial of degree $n-1$. 
\end{thm}

\paragraph{Trees of spaces.} The first step of our proof of Theorem~\ref{thm:BigIntro} is, given a right-angled Artin group $A(\Gamma)$, to fix a complete-cut-decomposition of $\Gamma$, to decompose $A(\Gamma)$ accordingly as a graph of groups, and to think geometrically of $A(\Gamma)$ as the corresponding tree of spaces. Below, we fix our definition of tree of spaces for clarity, and we record a couple of preliminary results.

\begin{definition}
Let $X$ be a metric space. A decomposition $(T, (V_s)_{s \in V(T)}, (E_a)_{a \in E(T)})$ of $X$ as a \emph{tree of spaces} is the data of a tree $T$ and two collections of connected subspaces $(V_s)_{s \in V(T)}$ and $(E_a)_{a \in E(T)}$ respectively indexed by the vertex- and edge-sets of $T$ such that, for every edge $\{r,s\} \in E(T)$, $V_r \cap V_s = E_{\{r,s\}}$. 
\end{definition}

\noindent
We start by observation that:

\begin{lemma}\label{lem:SepTreeSpaces}
Let $X= (T, (V_s)_{s \in V(T)}, (E_a)_{a \in E(T)})$ be a tree of spaces and $Y \leq X$ a coarsely connected subspace. Fix an $s \in V(T)$ an assume that $V_s$ is not contained in the neighbourhood of $E_a$ for some edge $a \in E(T)$ with $s$ as an endpoint. If $Y$ is not coarsely separated by $\{ E_a \mid s \text{ is an endpoint of } a\}$, then 
\begin{itemize}
	\item either $Y$ is contained in a neighbourhood of $V_s$, 
	\item or there exists an edge $a \in E(T)$ having $s$ as an endpoint such that $E_a$ coarsely separates $Y$ and~$V_s$. 
\end{itemize}
\end{lemma}

\begin{proof}
We start by introducing some notation. In $T$, let $\{a_i \mid i \in I\}$ denote the edges having $s$ as an endpoint. For short, we denote by $A_i$ the edge-space indexed by $a_i$ for every $i \in I$. Removing the vertex $s$ from $T$ yields branches indexed by $I$; we denote by $B_i$ the branch delimited by $a_i$ for every $i \in I$. Finally, for every $i \in I$, let $B_i^+ \subset X$ denote the union of all the vertex-spaces indexed by the vertices of $B_i$. 

\medskip \noindent
Assume that $Y$ is not contained in a neighbourhood of $V_s$. So, for every $n \geq 1$, there exists $y_n \in Y$ such that $d(y_n,V_s) \geq n$. Notice that there exists some $i \in I$ such that $y_n \in B_i^+$ for all but finitely many $n$. 

\medskip \noindent
Otherwise, we would find, for every $L \geq 0$, two distinct indices $i \neq j$ in $I$ and two integers $n,m \geq L$ such that $y_n \in B_i^+$ and $y_m \in B_j^+$. Clearly, $y_n$ and $y_m$ are separated by $E_i$ and the balls $B(y_n,L-1)$ and $B(y_m,L-1)$ are disjoint from $E_i$. This implies that $Y$ is coarsely separated by $\{ A_i \mid i \in I\}$, contrary to our assumption.

\medskip \noindent
Now, we claim that $A_i$ coarsely separates $Y$ from $V_s$. Because $Y$ is connected and contains points in $B_i^+$ arbitrarily far away from $V_s$, and a fortiori from $A_i \subset V_s$, we know that $Y$ must be contained in a neighbourhood of $B_i^+$, since otherwise $Y$ would be coarsely separated by $A_i$ (and a fortiori by $\{ A_i \mid i \in I\}$). On the other hand, since $V_s$ is not contained in a neighbourhood of $A_i$ by assumption, we know that it contains points arbitrarily far away from $A_i$. Thus, due to the fact that $A_i$ separates $B_i^+\backslash V_s$ and $V_s\backslash A_i$, we conclude that $A_i$ coarsely separates $Y$ and $V_s$, as desired.
\end{proof}

\noindent
We will also need the following statement, proved in our previous work \cite{CoarseRAAG}.

\begin{prop}[\cite{CoarseRAAG}]\label{prop:SepTreeSpaces}
Let $X$ be a tree of spaces and $Y \leq X$ a subspace. If $Y$ coarsely separates $X$, then either it coarsely separates some vertex-space or it contains some edge-space in a neighbourhood. 
\end{prop}

\section{Proof of Theorem~\ref{thm:BigIntro}}

\noindent
In order to prove Theorem~\ref{thm:BigIntro}, we start by extracting from \cite{CoarseRAAG} a weak version of our theorem:

\begin{thm}\label{thm:AbSepRAAG}
Let $\Gamma$ be a finite graph. The following assertions are equivalent:
\begin{itemize}
	\item[(i)] $A(\Gamma)$ contains a coarsely separating abelian subgroup;
	\item[(ii)] $A(\Gamma)$ virtually splits over an abelian subgroup;
	\item[(iii)] $A(\Gamma)$ splits over an abelian subgroup;
	\item[(iv)] $\Gamma$ contains a separating complete subgraph or is complete.
\end{itemize}
\end{thm}

\begin{proof}
If $\Gamma$ is complete, then $A(\Gamma)$ is free abelian, say of rank $n \geq 1$, and it splits as an HNN extension over a free abelian subgroup of rank $n-1$. And, if $\Gamma$ contains a separating complete subgraph $\Lambda$, then $A(\Gamma)$ clearly splits over $\langle \Lambda \rangle$, which is a abelian. Thus, the implication $(iv) \Rightarrow (iii)$ holds. The implications $(iii) \Rightarrow (ii) \Rightarrow (i)$ are clear. The implication $(i) \Rightarrow (iv)$ is a particular case of the main result of \cite{CoarseRAAG}. 
\end{proof}

\noindent
Given a right-angled Artin group $A(\Gamma)$ defined by a graph $\Gamma$, it is clear that, if $\Lambda$ is a separating subgraph of $\Gamma$, then $A(\Gamma)$ splits over $\langle \Lambda \rangle$. More precisely, we can describe $\Gamma$ as the union of two subgraphs containing $\Lambda$ properly, say $\Gamma_1$ and $\Gamma_2$, amalgamated along $\Lambda$. Then, it is clear that $A(\Gamma)$ splits as the amalgamated product $\langle \Gamma_1 \rangle \ast_{\langle \Lambda \rangle} \langle \Gamma_2 \rangle$. Our next lemma exhibits another source of splittings. 

\begin{lemma}\label{lem:NonObviousSplit}
Let $\Gamma$ be a graph and $u \in V(\Gamma)$ a vertex satisfying $\mathrm{star}(u) \neq \Gamma$. The group $A(\Gamma)$ splits over a subgroup isomorphic to $\langle \mathrm{star}(u) \rangle$.
\end{lemma}

\begin{proof}
Let $\Gamma_1$ and $\Gamma_2$ be two copies of $\Gamma$. For every vertex $v \in V(\Gamma)$, we denote by $v_1$ (resp.\ $v_2$) the copy of $v$ in $\Gamma_1$ (resp.\ $\Gamma_2$). We amalgamate the right-angled Artin groups $A(\mathrm{star}(u_1))$ and $A(\Gamma_2)$ along the subgroup $\langle \mathrm{star}(u) \rangle$ through the embeddings given by 
$$\varphi_1 : v \in V(\mathrm{star}(u)) \mapsto \left\{ \begin{array}{l}   v_1^2 \text{ if } v=u \\  v_1 \text{ otherwise} \end{array} \right. \text{ and } \varphi_2 : v \in V(\mathrm{star}(u)) \mapsto v_2.$$
Notice that $\varphi_1$ is not surjective (its image does not contain $u_1$) and neither is $\varphi_2$ as soon as $\Gamma \neq \mathrm{star}(u)$, so our splitting is not trivial. Our amalgamate admits as a presentation
\small$$\left\langle \begin{array}{c} v_1 \ (v \in V(\mathrm{star}(u))) \\ v_2 \ (v \in V(\Gamma)) \end{array} \mid \begin{array}{c} \left[v_1,w_1\right]=1 \text{ if } \{v,w\} \in E(\mathrm{star}(u)) \\ \left[v_2,w_2\right]=1 \text{ if } \{v,w\} \in E(\Gamma) \end{array}, \begin{array}{c} u_1^2=u_2 \\ v_1=v_2 \ (v \in V(\mathrm{link}(u))) \end{array} \right\rangle.$$\normalsize
By using the rightmost relations, we can remove from the presentation the generators $v_2$ for $v \in V(\mathrm{star}(u))$, which yields:
\small$$\left\langle \begin{array}{c} v_1 \ (v \in V(\mathrm{star}(u))) \\ v_2 \ (v \in V(\Gamma \backslash \mathrm{star}(u))) \end{array} \mid \begin{array}{c} \left[v_1,w_1\right] \text{ if } v,w \in V(\mathrm{star}(u)) \text{ adjacent} \\ \left[v_2,w_2\right]=1 \text{ if } v,w \in V(\Gamma \backslash \mathrm{star}(u)) \text{ adjacent} \\ \left[v_1,w_2\right]=1 \text{ if } v \in V(\mathrm{link}(u)) \text{ and } w \in V(\Gamma \backslash \mathrm{star}(u)) \text{ adjacent} \\ \left[u_1^2,v_1\right] = 1 \text{ for every } v \in V(\mathrm{link}(u)) \end{array} \right\rangle.$$\normalsize
Notice that the relations $[u_1^2,v_1]=1$ for $v \in V(\mathrm{link}(u))$ are consequences of the relations $[v_1,w_1]=1$ for adjacent $v,w \in V(\mathrm{star}(u))$, so they can be removed from the presentation. Then, our presentation is just the canonical presentation of $A(\Gamma)$. 
\end{proof}

\noindent
Before turning to the proof of Theorem~\ref{thm:BigIntro}, we need a last elementary observation:

\begin{lemma}\label{lem:GeomCont}
Let $G$ be a finitely generated group, $g \in G$ an element, and $H,K \leq G$ two subgroups. If $gH$ is contained in a neighbourhood of $K$, then $gHg^{-1}$ is virtually contained in $K$.
\end{lemma}

\begin{proof}
Let $D \geq 0$ be such that $gH$ is contained in the $D$-neighbourhood of $K$ and let 
$$\mathcal{C}: = \{ \text{coset } hK \mid gH \subset D\text{-neighbourhood of } hK\}.$$
Notice that the cosets in $\mathcal{C}$ are pairwise disjoint and they all intersect the ball $B(g,D)$ in $G$. Therefore, $\mathcal{C}$ is finite. Since $gHg^{-1}$ stabilises $gH$, necessarily $\mathcal{C}$ is $gHg^{-1}$-invariant. The kernel of the action $gHg^{-1} \curvearrowright \mathcal{C}$ then yields a finite-index subgroup of $gHg^{-1}$ that stabilises $K$, i.e.\ that is contained in $K$. 
\end{proof}

\begin{proof}[Proof of Theorem~\ref{thm:BigIntro}.]
If $\Gamma$ is a complete graph with $n+1$ vertices, then $A(\Gamma)$ is a free abelian group of rank $n+1$, which can also be described as an HNN extension over $\mathbb{Z}^n$. Next, if $\Gamma$ contains a complete subgraph $\Lambda$ with $n$  vertices and with a subgraph $\Lambda_0$ that separates $\Gamma$, then we show that $A(\Gamma)$ splits over $\mathbb{Z}^n$ as follows. Because $\Lambda_0$ separates $\Gamma$, we can described $\Gamma$ as a union of subgraphs $\Gamma_0, \Gamma_1, \ldots, \Gamma_k$ ($k \geq 1$) amalgamated along $\Lambda_0$. We assume without loss of generality that each $\Gamma_i$ contains $\Lambda_0$ properly. Since $\Lambda$ is complete, it must be contained entirely into some $\Gamma_i$, say $\Gamma_0$ up to reindexing the $\Gamma_i$. If $\Lambda \subsetneq \Gamma_0$, then $\Lambda$ separates $\Gamma_0 \backslash \Lambda$ and $\Gamma_1 \backslash \Lambda$, which clearly implies that $A(\Gamma)$ splits over $\langle \Lambda \rangle \simeq \mathbb{Z}^n$. So we can assume that $\Lambda= \Gamma_0$. If $\Lambda_0 \subsetneq \Lambda$, then, given an arbitrary vertex $v \in V(\Lambda \backslash \Lambda_0)$, it follows from Lemma~\ref{lem:NonObviousSplit} that $A(\Gamma)$ splits over a subgroup isomorphic to $\langle \mathrm{star}(v) \rangle = \langle \Lambda \rangle \simeq \mathbb{Z}^n$. If $\Lambda_0=\Lambda$, then $\Lambda$ is separating itself and we deduce that $A(\Gamma)$ again splits over $\langle \Lambda \rangle \simeq \mathbb{Z}^n$, as desired. 

\medskip \noindent
Thus, we have proved the implication $(iv) \Rightarrow (iii)$. The implications $(iii) \Rightarrow (ii) \Rightarrow (i)$ are clear. Our goal now is to prove $(i) \Rightarrow (iv)$. So, from now on, we assume that $A(\Gamma)$ contains a coarsely separating abelian subgroup of rank $n$, say $C$. 

\medskip \noindent
Fix a complete-cut-decomposition $(T, (V_s)_{s \in V(T)})$ of $\Gamma$ as given by Proposition~\ref{prop:CCD}. This induces a decomposition of $A(\Gamma)$ as a graph of groups 
$$\mathcal{G}:= (T, (\langle V_s \rangle)_{s \in V(T)}, (E_{\{r,s\}}:=\langle V_r \cap V_s \rangle)_{\{r,s\} \in E(T)})$$ 
where, for every $\{r,s\} \in E(T)$, the monomorphism $\langle V_r \cap V_s \rangle \hookrightarrow V_s$ is induced by the inclusion map $V_r \cap V_s \hookrightarrow V_s$. 

\medskip \noindent
If our complete-cut-decomposition of $\Gamma$ is trivial, i.e.\ $T$ is reduced to a single vertex, then it follows from Theorem~\ref{thm:AbSepRAAG} that $\Gamma$ is a complete graph. Thus, $A(\Gamma)$ is a free abelian group of rank $|V(\Gamma)|$ containing a coarsely separating subgroup of rank $n$ (namely, $C$), hence $|V(\Gamma)|=n+1$ according to Lemma~\ref{lem:CodimAb}. We conclude, as desired, that $\Gamma$ is a complete graph with $n+1$ vertices. From now on, we assume that our complete-cut-decomposition of $\Gamma$ is non-trivial.

\medskip \noindent
First of all, assume that $C$ does not coarsely separate any $g \langle V_s \rangle$ for $g \in A(\Gamma)$ and $s \in V(T)$. It follows from Proposition~\ref{prop:SepTreeSpaces} that there exist $g \in A(\Gamma)$ and $a \in E(T)$ such that $g \langle E_a \rangle$ is contained in a neighbourhood of $C$. According to Lemma~\ref{lem:GeomCont}, this implies that $g \langle E_a \rangle g^{-1}$ is virtually contained in $C$. Then, we deduce from Proposition~\ref{prop:ArtinSubAbelian} that $E_a$ is (a complete cut) contained in a complete subgraph with $n$ vertices, leading to the desired conclusion.

\medskip \noindent
Next, assume that $C$ does coarsely separate $g \langle V_s \rangle$ for some $g \in A(\Gamma)$ and $s \in V(T)$. Because $V_s$ has no complete cut, it follows from Theorem~\ref{thm:AbSepRAAG} that $V_s$ is a complete graph. It follows from Theorem~\ref{thm:SepEucGrowth} that $V_s$ has $\leq n+1$ vertices. 

\noindent
\begin{minipage}{0.4\linewidth}
\includegraphics[width=0.9\linewidth]{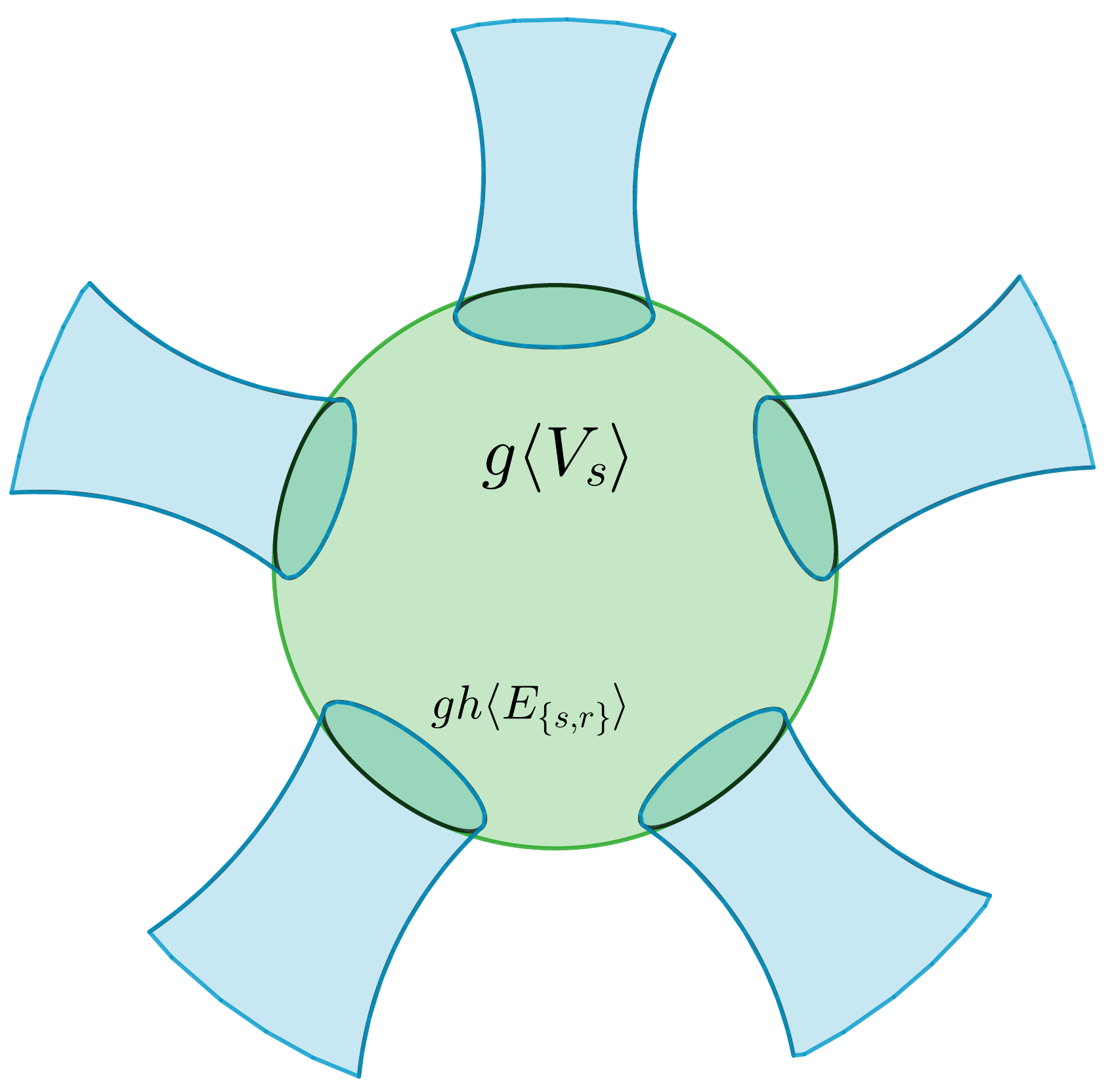}
\end{minipage}
\begin{minipage}{0.58\linewidth}
Now, we distinguish two cases, depending on whether or not $C$ is coarsely separated by 
$$\mathcal{Z}:= \{ gh E_{\{s,r\}} \mid h \in \langle V_s \rangle, \ r \in V(T) \text{ neighbour of } s\}.$$
Notice that $\mathcal{Z}$ is just the edge-spaces tangent to $g \langle V_s \rangle$ in the graph-of-spaces decomposition of $A(\Gamma)$ given by our graph-of-groups decomposition.
\end{minipage}

\medskip \noindent
If $C$ is coarsely separated by $\mathcal{Z}$, then Theorem~\ref{thm:SepEucGrowth} implies that $\mathcal{Z}$ has growth at least polynomial of degree $n-1$. But $\mathcal{Z}$  has polynomial growth of degree $\max\{ |V(E_{\{s,r\}})| \mid r \text{ neighbour of } s\}$, so there must exist some neighbour $r$ of $s$ such that $E_{\{r,s\}}=V_r \cap V_s$ has at least $n-1$ vertices. But we also know that $V_r \cap V_s$ is properly contained in $V_s$, so it must contain less vertices than $V_s$, hence at most $n$. Thus, $V_r \cap V_s$ has either $n$ or $n-1$ vertices. In the former case, $V_r \cap V_s$ is a complete cut of $\Gamma$ of size $n$, and there is nothing else to prove. And, in the latter case, because $V_s$ contains properly $V_r \cap V_s$, it must be a complete graph with $\geq n$ vertices containing a subgraph with $n-1$ vertices that separates $\Gamma$. Thus, the complete cut $V_r \cap V_s$ is contained in some complete subgraph in $V_s$ with $n$ vertices. 

\medskip \noindent
Now, assume that $C$ is not coarsely separated by $\mathcal{Z}$. According to ~\ref{lem:SepTreeSpaces}, two cases may happen. 

\medskip \noindent
First, $C$ may be contained in a neighbourhood of $g \langle V_s \rangle$. Then, $\langle V_s \rangle$ must have growth at least polynomial of degree $n$, which implies that $V_s$ has $\geq n$ vertices. But we already know that $V_s$ has $\leq n+1$ vertices. So either $V_s$ has $n$ vertices, providing a complete subgraph of size $n$ containing some complete cut of $\Gamma$ (namely, $E_{\{s,r\}}$ for any neighbour $r$ of $s$); or $V_s$ has $n+1$ vertices, in which case, given some neighbour $r$ of $s$, the complete cut $E_{\{s,r\}}$ either has $n$ vertices or is contained in some complete subgraph of $V_s$ with $n$ vertices. 

\medskip \noindent
\begin{minipage}{0.4\linewidth}
\includegraphics[width=0.9\linewidth]{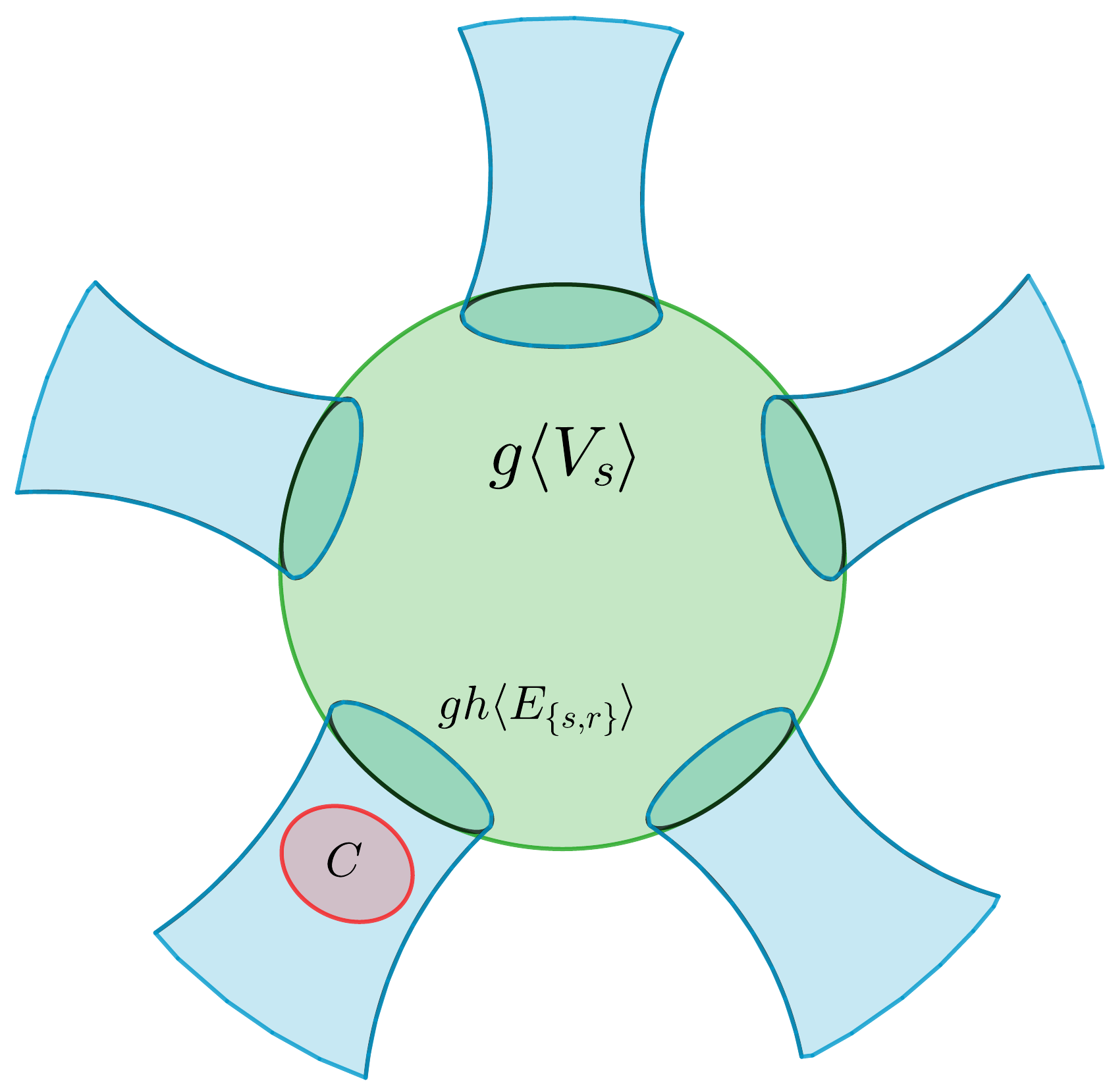}
\end{minipage}
\begin{minipage}{0.58\linewidth}
Second, $C$ and $g \langle V_s \rangle$ may be coarsely separated by $ghE_{\{s,r\}}$ for some $h \in \langle V_s \rangle$ and some neighbour $r$ of $s$. We claim that our complete cut $E_{\{s,r\}}$ is contained ins some complete subgraph of $\Gamma$ with $n$ vertices, which will complete the proof of our theorem. 
\end{minipage}

\medskip \noindent
On the one hand, if $k$ denotes the number of vertices of $V_s$, then we know that $E_{\{s,r\}}$ has $\leq k-1$ vertices. On the other hand, since the coarse intersection between $C$ and $g \langle V_s \rangle$ coarsely separates $\langle g V_s \rangle$, it growth must be at least polynomial of degree $k-1$. Because this coarse intersection is also contained in a neighbourhood of $gh \langle E_{\{s,r\}} \rangle$, it follows that $E_{\{s,r\}}$ must have at least $k-1$ vertices. We conclude that $E_{\{r,s\}}$ has exactly $k-1$ vertices. 

\medskip \noindent
Note that $g\langle V_s \rangle$ decomposes as a product $gh \langle E_{\{r,s\}} \rangle \times \mathbb{Z}$. Since, as already said, $gh \langle E_{\{r,s\}} \rangle$ coarsely contains the coarse intersection betwenn $C$ and $g \langle V_s \rangle$, coarse intersection that coarsely separates $g \langle V_s \rangle$, we deduce from ~\ref{lem:CoarseSepEuclidean} that our coarse intersection is quasi-dense in $gh \langle E_{\{r,s\}}$. As a consequence, $gh \langle E_{\{r,s\}}$ is contained in a neighbourhood of $C$, which implies, according to Lemma~\ref{lem:GeomCont}, that $gh \langle E_{\{r,s\}} \rangle h^{-1}g^{-1}$ is virtually contained in $C$. We conclude thanks to Proposition~\ref{prop:ArtinSubAbelian} that the complete cut $E_{\{r,s\}}$ of $\Gamma$ is contained in a complete subgraph with $n$ vertices. 
\end{proof}

\addcontentsline{toc}{section}{References}

\bibliographystyle{alpha}
{\footnotesize\bibliography{VirtSplitRAAG}}

\Address

%

\end{document}